\documentclass{birkjour}
 \newtheorem{thm}{Theorem}[section]

 \theoremstyle{definition}
 
 \theoremstyle{remark}

 \numberwithin{equation}{section}

\begin{document}

%-------------------------------------------------------------------------
% editorial commands: to be inserted by the editorial office
%
%\firstpage{1} \volume{228} \Copyrightyear{2004} \DOI{003-0001}
%\seriesextraline{This is the Concrete Title of this Book\br H.E. R and S.T.C. W, Eds.}
%\firstpage{1}

\title[HORADAM OCTONIONS]
 {HORADAM OCTONIONS}

%----------Author 1
\author[Adnan KARATA\c{S}]{Adnan KARATA\c{S}}

\address{%
Pamukkale University,\\
Faculty of Arts and Sciences,\\
Department of Mathematics,\\
Denizli/TURKEY}

\email{adnank@pau.edu.tr}
%----------Author 2
\author[Serp\.{i}l HALICI]{Serp\.{i}l HALICI}

\address{%
Pamukkale University,\\
Faculty of Arts and Sciences,\\
Department of Mathematics,\\
Denizli/TURKEY}

\email{shalici@pau.edu.tr}



\date{November 4, 2016}

\begin{abstract}
In this paper, first we define Horadam octonions by Horadam sequence which is a generalization of second order recurrence relations.  Also, we give some fundamental properties involving the elements of that sequence. Then, we obtain their Binet-like formula, ordinary generating function and Cassini identity. 
\end{abstract}

%%% ----------------------------------------------------------------------
\maketitle
%%% ----------------------------------------------------------------------
%\tableo.fcontents
\vspace{1.5 cm}
\section{Introduction}
Four dimensional quaternion algebra is defined by W. R. Hamilton in 1843. After the definition of quaternion algebra, J. T. Graves asked Hamilton about higher dimensional algebras. He tried to define eight and sixteen dimensional algebras, but Graves encountered zero divisors on sixteen dimensional algebra. Because of zero divisors, Graves defined only eight dimensional real octonion algebra in 1843. Besides Graves' definition, A. Cayley defined also octonions in 1845. For detailed knowledge reader referred to [2]. Octonion algebra is eight dimensional, non-commutative, non-associative and normed division algebra. This algebra is used in various physical problems such as; supersymmetry, super gravity and super strings. In addition octonions are  used in many subjects such as, unified fields of dyons and gravito-dyons, electrodynamic, electromagnetism . For background the reader can find their properties in [2, 4, 5, 6, 7, 8].  \\~\\ 
The octonion algebra $\mathbb{O}$ form an eight-dimensional real algebra with a basis $\{1,e_1,e_2,$ $ \ldots, e_7 \}$. Addition operation is component-wise and multiplication operation of base elements can be made by Cayley-Dickson process with quaternions or Fano plane or multiplication table. For detailed information, see [2].\\~\\
Now, we give some fundamental definitions for the octonion algebra. Let $x$ be any octonion, i.e;
$$ x= \sum^7 _{i=0}{k_i e_i}, \,\ \,\ k_i \in \mathbb{R}$$ 
where $e_0, e_1, \ldots, e_7$ are base elements. Conjugate of any octonion $x$ can be defined as $$\overline{x}=k_0 - \sum^7 _{i=1}{k_i e_i}.$$ One can easily see that the conjugate of an octonion $x$ is an involution of first kind [19]. With the help of the conjugate, norm of an octonion $x$ can be defined as follows,
$$Nr(x)=\overline{x} \circ x=x \circ \overline{x}=k_0^2+k_1^2+ \dots + k_7^2. $$
To reader who wants to make further readings on octonions, we refer to [2, 7, 20].\\~\\
One can find many studies involving sequences with positive integers which satisfy a recursive relation in the literature. And many paper are dedicated to Fibonacci sequences and their generalization.\\
In this study we generalize Fibonacci octonions and we investigate their fundamental properties. In first section, we give definition of octonions. In second section, we give definition of Horadam numbers which is generalization of Fibonacci numbers. In third section, we introduce Horadam octonions which contains Fibonacci, Lucas, Pell, Jacobsthal and all other second order sequences. For this newly defined sequence we give the generating function, norm value, Cassini identity, a summation formula. As we give identities and properties mentioned above we present connections to earlier studies.
\section{\textbf{HORADAM NUMBERS}}
The famous Fibonacci numbers are second order recursive relation and used in various disciplines. Some lesser known second order recursive relations are Lucas numbers, Pell numbers, Jacobsthal numbers, etc.. Many mathematicians tried to generalize these second order recursive relations. Then Horadam generalized these relations and this generalization is named as Horadam sequence. Horadam sequences are firstly defined for generalization of Fibonacci and Lucas recurrences as $U_n=U_{n-1}+U_{n-2}$ in [13] and this version of Horadam numbers studied by Koshy in [18]. Then this definition is altered such that it includes other integer sequences like Jacobstal numbers. We note that there are many identities and properties about Horadam sequence [14, 15, 16].\\\\
Now, let us recall the definition of Horadam numbers. In [15], Horadam numbers are defined as \\
\begin{equation} \{ w_n(a,b;p,q) \}: w_n = p w_{n-1} + q w_{n-2}; w_0=a, w_1=b, \,\ \,\  (n \geq 2) \end{equation}
where $a,b,p,q$ are integers. Let us give four important properties that are needed. Firstly, Binet formula of Horadam sequence can be calculated using its characteristic equation. The roots of characteristic equation are
\begin{equation} \alpha = \frac{p+\sqrt{p^2+4q}}{2}, \,\ \,\ \beta= \frac{p-\sqrt{p^2+4q}}{2}. \end{equation}
Using these roots and the recurrence relation Binet formula can be given as follows
\begin{equation} w_n=\frac{A\alpha^n-B\beta^n}{\alpha - \beta}, \,\ \,\ \,\ A=b-a\beta \mbox{ and }B=b-a \alpha.\end{equation}
Secondly, the generating function for Horadam numbers is
\begin{equation} g(t)=\frac{w_0+(w_1-pw_0)t}{1-pt-qt^2}. \end{equation}
Thirdly, the Cassini identity for Horadam numbers is
\begin{equation} w_{n+1}w_{n-1}-w^2_n=q^{n-1}(pw_0 w_1- w_1^2-w_0^2q). \end{equation}
Lastly, a summation formula for Horadam numbers is
\begin{equation} \sum_{i=0}^n{w_i}=\frac{w_1-w_0 (p-1)+qw_n - w_{n+1}}{1-p-q}.\end{equation}
Horadam firstly defined Horadam numbers on $\mathbb{R}$ and then defined  Horadam numbers on $\mathbb{C}$ and $\mathbb{H}$ [14]. In addition, Halici gave a very complete survey about Horadam quaternions in [9].\\
More recently octonions have been studied by many authors. For example, the Fibonacci octonions, Pell octonions and Modified Pell octonions appeared in [3, 17, 22]. There are many studies about Fibonacci numbers over dual octonions and generalized octonions [11, 21, 23]. Also, some representations of Fibonacci octonions are considered, for example see [12]. \\
In the next section we define a new octonion sequence with Horadam components which is generalization of earlier studies, for example, see [3, 11, 12, 17, 21, 22, 23]. And then we give some fundamental properties and identities related with this sequences.
\section{\textbf{HORADAM OCTONIONS}}
In this section we introduce Horadam octonions. Horadam octonions are octonions with coefficients from Horadam sequence and we inspired the idea from [9]. In [9], author studied Binet formula, generating function, Cassini identity, summation formula and norm value for Horadam quaternions. Similarly we define Horadam octonions and we investigate Binet formula, generating function, Cassini identity, summation formula and norm value for Horadam octonions. Horadam octonions can be shown as follows
$$\mathbb{O}G_{n}=w_n e_0 + w_{n+1} e_1 + \dots + w_{n+7} e_7$$
where $w_n$ as in equation $(2.1)$. Because of its coefficients this new octonion can be called Horadam octonion. After some necessary
calculations we acquire the following recurrence relation;
$$\mathbb{O}G_{n+1}=p\mathbb{O}G_{n}+q\mathbb{O}G_{n-1}.$$
In this section we give Binet formula, generating function, Cassini identities, summation and norm of these octonions. Binet formula is very useful for finding desired Horadam octonion and this formula takes part at many theorems' proof.\\
The Cassini identity yields many fascinating by-products and this formula is used to establish interesting results concerning with some sequences.\\
\begin{thm}
The Binet formula for Horadam octonions is \\
\begin{equation}\mathbb{O}G_{n}=\frac{A\underline{\alpha}\alpha^n-B\underline{\beta}\beta^n}{\alpha - \beta}, \end{equation} \\
where $\underline{\alpha}=1e_0 + \alpha e_1+ \alpha^2 e_2+ \dots + \alpha^7 e_7$ and $\underline{\beta}=1e_0 + \beta e_1+ \beta^2 e_2+ \dots + \beta^7 e_7.$
\end{thm}
\begin{proof}
Binet formula for Horadam octonions can be calculated similar to Binet formula for Horadam sequence. By using characteristic equation
\begin{equation}t^2-pt-q=0.\end{equation}
The roots of characteristic equation is
\begin{equation} \alpha = \frac{p+\sqrt{p^2+4q}}{2}, \,\ \,\ \beta= \frac{p-\sqrt{p^2+4q}}{2}. \end{equation}
Using these roots and the recurrence relation Binet formula can be written as follows
\begin{equation}\mathbb{O}G_{n}=\frac{A\underline{\alpha}\alpha^n-B\underline{\beta}\beta^n}{\alpha - \beta}\end{equation}
where
\begin{equation}\underline{\alpha}=1e_0 + \alpha e_1+ \alpha^2 e_2+ \dots + \alpha^7 e_7\mbox{ and }\underline{\beta}=1e_0 + \beta e_1+ \beta^2 e_2+ \dots + \beta^7 e_7.\end{equation}
\end{proof}
In the Binet formula of Horadam octonions if we take $A=B=1$ and calculating the value of $\alpha, \beta$ then we have $(3.2)$ \\
\begin{equation} MP_n =\frac{\underline{\alpha}\alpha^n+\underline{\beta}\beta^n}{\alpha - \beta}.\end{equation}
We should note that the equation $(3.6)$ is Binet formula of Modified Pell octonions which is given by Catarino in [3]. Also, if we take $A=B=1,$ and calculate the value $\alpha, \beta$ with respect to its recurrence relation, we have Binet formula for Fibonacci octonions which is given by Ke\c{c}ilio\u{g}lu and Akku\c{s} in [17].\\
In the following theorem we give the generating function for Horadam octonions.
\begin{thm}
The generating function for Horadam octonions is \\
\begin{equation}\frac{\mathbb{O}G_{0}+(\mathbb{O}G_{1}-p\mathbb{O}G_{0})t}{1-pt-qt^2}.\end{equation}
\end{thm}
\begin{proof}
To prove this claim, firstly, we need to write generating function for Horadam octonions;
\begin{equation}g(t)=\mathbb{O}G_{0}t^0+ \mathbb{O}G_{1}t +\mathbb{O}G_{2}t^2+ \dots+ \mathbb{O}G_{n}t^n+\dots \end{equation} \\
Secondly, we need to calculate $ptg(t)$ and $qt^2g(t)$ as the following equations;\\
\begin{equation}ptg(t)=\sum^\infty_{n=0}{p\mathbb{O}G_{n}t^{n+1}} \mbox{ and } qt^2g(t)=\sum^\infty_{n=0}{q\mathbb{O}G_{n}t^{n+2}}. \end{equation}\\
Finally, if we made necessary calculations, then we get the following equation  \\
\begin{equation}g(t)=\frac{\mathbb{O}G_{0}+(\mathbb{O}G_{1}-p\mathbb{O}G_{0})t}{1-pt-qt^2}\end{equation}which is the generating function for Horadam octonions
\end{proof}
Using the initial values for Modified Pell octonions and the equation $(3.7)$, we obtain
\begin{equation} g(t)=\frac{MPO_{0}+(MPO_{1}-2MPO_{0})t}{1-2t-t^2}.\end{equation}
In fact the formula $(3.11)$ is the generating function for Modified Pell octonions given by Catarino in [3]. In addition if we take initial values of Fibonacci octonions and calculate equation $(3.7)$ accordingly then we get generating function for Fibonacci octonions. So, generating function for Horadam octonions generalizes Fibonacci octonions, Lucas, octonions, Modified Pell octonions, etc.. \\\\
In the following theorem, we state two different Cassini identities which occur from non-commutativity of octonion multiplication.\\
\begin{thm}
For Horadam octonions the Cassini formulas are as follows
\begin{equation} i) \,\ \mathbb{O}G_{n-1} \circ \mathbb{O}G_{n+1}-\mathbb{O}G^2_{n}=\frac{AB(\alpha \beta)^{n-1}(\beta \underline{\alpha}\underline{\beta}-\alpha \underline{\beta}\underline{\alpha})}{\alpha-\beta}, \end{equation}
\begin{equation} ii) \,\  \mathbb{O}G_{n+1} \circ \mathbb{O}G_{n-1}-\mathbb{O}G^2_{n}=\frac{AB(\alpha \beta)^{n-1}(\beta \underline{\beta}\underline{\alpha}-\alpha \underline{\alpha}\underline{\beta})}{\alpha-\beta}. \end{equation}
\end{thm}
\begin{proof}
We use Binet formula in order to prove equation $(3.12)$ \\
$$\mathbb{O}G_{n+1} \circ \mathbb{O}G_{n-1}-\mathbb{O}G^2_{n}=$$\
$$= \frac{A\underline{\alpha}\alpha^{n-1}-B\underline{\beta}\beta^{n-1}}{\alpha - \beta} \frac{A\underline{\alpha}\alpha^{n+1}-B\underline{\beta}\beta^{n+1}}{\alpha - \beta}- \bigg( \frac{A\underline{\alpha}\alpha^{n}-B\underline{\beta}\beta^{n}}{\alpha - \beta} \bigg)^2.$$ \\
If necessary calculations are made, we obtain \\
$$\mathbb{O}G_{n+1} \circ \mathbb{O}G_{n-1}-\mathbb{O}G^2_{n}=\frac{AB(\alpha \beta)^{n-1}(\beta \underline{\beta}\underline{\alpha}-\alpha \underline{\alpha}\underline{\beta})}{\alpha-\beta}$$ \\
which is desired. In a similar way, the equation $(3.13)$ can be easily obtain. 
\end{proof}
Cassini identities are studied for Fibonacci octonions in [12, 17] and for Modified Pell octonions in [3]. \\\\
If we consider as $a=0,b=1$ and $p=1, q=1$ in the equation $(3.12)$ and $(3.13)$ we have the Cassini identities of Fibonacci octonions. Hence, we conclude that Horadam octonions are a generalization of all the Fibonacci-like octonions.\\
Summation formula for the first $n+1$ Horadam octonions can be given as follows. \\ 
\begin{thm}
The sum formula for Horadam octonions is follows, \\
\begin{equation} \sum_{i=0}^n{\mathbb{O}G_i}=\frac{1}{\alpha-\beta}\big( \frac{B\underline{\beta}\beta^{n+1}}{1-\beta}-\frac{A\underline{\alpha}\alpha^{n+1}}{1-\alpha} \big) + K,\end{equation} \\
where $K$ is as follows, \\
$$K=\frac{A\underline{\alpha}(1-\beta)-B\underline{\beta}(1-\alpha)}{(\alpha-\beta)(1-\alpha)(1-\beta)}.$$
\end{thm}
\begin{proof}
Using the Binet formula we can calculate the summation formula as follows
$$\sum_{i=0}^n{\mathbb{O}G_i}=\frac{A\underline{\alpha}\alpha^{n}-B\underline{\beta}\beta^{n}}{\alpha-\beta}=\frac{A\underline{\alpha}}{\alpha-\beta} \sum_{i=0}^n \alpha^{n}-\frac{B\underline{\beta}}{\alpha-\beta} \sum_{i=0}^n \beta^{n} $$
From geometric series we get the following
$$\sum_{i=0}^n{\mathbb{O}G_i}=\frac{A\underline{\alpha}}{\alpha-\beta}\sum_{i=0}^n \frac{1-\alpha^{n+1}}{1-\alpha}-\frac{B\underline{\beta}}{\alpha-\beta} \sum_{i=0}^n \frac{1-\beta^{n+1}}{1-\beta}. $$
After direct calculations we will get the explicit result as 
$$\sum_{i=0}^n{\mathbb{O}G_i}=\frac{1}{\alpha-\beta}\big( \frac{B\underline{\beta}\beta^{n+1}}{1-\beta}-\frac{A\underline{\alpha}\alpha^{n+1}}{1-\alpha} \big) +\frac{A\underline{\alpha}(1-\beta)-B\underline{\beta}(1-\alpha)}{(\alpha-\beta)(1-\alpha)(1-\beta)}.$$
\end{proof}
Summation formula for Horadam octonions is a generalization of summation formula for Fibonacci octonions which can be calculated according to its initial values as follows.
$$\sum_{i=0}^n{\mathbb{O}_i}=\frac{1}{\sqrt{5}}\bigg( \frac{\underline{\beta}\beta^{n+1}}{1-\beta}-\frac{\underline{\alpha}\alpha^{n+1}}{1-\alpha} \bigg) -\frac{\underline{\alpha}(1-\beta)-\underline{\beta}(1-\alpha)}{\sqrt{5}}.$$

Using the initial values and roots of characteristic equation we state the norm of $nth$ Horadam octonion as follows.\\
\begin{thm}
The norm of nth Horadam octonion is
\begin{equation} \small Nr(\mathbb{O}G_n)=\frac{A^2\alpha^{2n}(1+\alpha^2+\alpha^4+\dots+\alpha^{14})+B^2\beta^{2n}(1+\beta^2+\beta^4+\dots+\beta^{14})}{(\alpha-\beta)^2}-L\end{equation}
where $L$ is
$$L=\frac{2AB(-q)^n(a+(-q)+\dots+(-q)^7)}{(\alpha- \beta)^2}.$$
\end{thm}
\begin{proof}
We defined $nth$ Horadam octonion as
$$\mathbb{O}G_{n}=w_n e_0 + w_{n+1} e_1 + \dots + w_{n+7} e_7.$$
So, norm of $nth$ Horadam octonion is
$$Nr(\mathbb{O}G_{n})=\mathbb{O}G_{n}\overline{\mathbb{O}G_{n}}=\overline{\mathbb{O}G_{n}}\mathbb{O}G_{n}=w_n^2+w_{n+1}^2+ \dots + w_{n+7}^2.$$
Making necessary calculations and using the equations $\alpha+\beta=p$ and $ \alpha\beta = -q$, we will get the explicit form of desired result as
$$ \small Nr(\mathbb{O}G_n)=\frac{A^2\alpha^{2n}(1+\alpha^2+\alpha^4+\dots+\alpha^{14})}{(\alpha-\beta)^2} + $$
$$\frac{B^2\beta^{2n}(1+\beta^2+\beta^4+\dots+\beta^{14})}{(\alpha-\beta)^2} - \frac{2AB(-q)^n(a+(-q)+\dots+(-q)^7)}{(\alpha- \beta)^2}.$$
\end{proof}
One can observe that the norm of Horadam octonions is generalization of Pell octonions and Pell-Lucas octonions. To get desired norms on [3] one can take values $a=0,b=1,p=2,q=1$ and $a=2,b=2,p=2,q=1$, respectively.
\maketitle
\section{\textbf{Conclusion}}
In this paper, we define Horadam octonions which is studied for the first time. Moreover, we give their Binet formula and some identities related with them. We demonstrate that our results are generalization of the other studies in this area. It should be noted that the further studies can be to investigated to get new identities about Horadam octonions.

% ------------------------------------------------------------------------

\begin{thebibliography}{99}

\bibitem{ 1 Akk} Akku\c{s} I. and O. Ke\c{c}ilio\u{g}lu \textit{Split Fibonacci and Lucas octonions.} Adv. Appl. Clifford Algebr. doi 10 (2015): s00006-014.

\bibitem{ 2 Baez} Baez, J. \textit{The octonions.} Bulletin of the American Mathematical Society 39.2 (2002): 145-205.

\bibitem{ 3 Cat} Catarino, P. \textit{The Modified Pell and Modified k-Pell Quaternions and Octonions.} Advances in Applied Clifford Algebras 26, (2016):577-590.

\bibitem{ 4 chan} Chanyal, B. C. \textit{Split octonion reformulation of generalized linear gravitational field equations.} Journal of Mathematical Physics 56(5), (2015): 051702. 

\bibitem{ 5 chan2}Chanyal, B. C., P. S. Bisht and O. P. S. Negi \textit{Generalized split-octonion electrodynamics.} International Journal of Theoretical Physics 50.6, (2011): 1919-1926.

\bibitem{ 6 chan3} Chanyal, B. C., S. K. Chanyal, \"{O}. Bekta\c{s} and S. Y\"{u}ce \textit{A new approach on electromagnetism with dual number coefficient octonion algebra.} International Journal of Geometric Methods in Modern Physics 13.09, (2016): 1630013.

\bibitem{ 7 Grsey1} G\"{u}rsey, F. and C. Tze. On the role of division, Jordan and related algebras in particle physics. World Scientific, 1996.

\bibitem{ 8 Grsey2} G\"{u}naydin, M. and F. G\"{u}rsey. \textit{Quark structure and octonions.} Journal of Mathematical Physics 14.11 (1973): 1651-1667.
\bibitem{ 9 S.H.} Halici, S. \textit{On a Generalization for Quaternion Sequences}, arXiv:1611.07660.
\bibitem{ 9 S.H} Halici, S. \textit{On Fibonacci quaternions.} Advances in Applied Clifford Algebras 22.2, (2012): 321-327.

\bibitem{ 10 Sh2} Halici, S. \textit{On dual Fibonacci octonions.} Advances in Applied Clifford Algebras 25.4 (2015): 905-914.

\bibitem{ 11 Halici} Halici, S. and A. Karata\c{s}. \textit{Some Matrix Representations of Fibonacci Quaternions and Octonions.} Advances in Applied Clifford Algebras: 1-10.

\bibitem{ 12 Horadam1} Horadam, A. F. \textit{A generalized Fibonacci sequence.} The American Math. Mont. 68.5, (1961):455-459.

\bibitem{ 13 Horadam} Horadam, A. F. \textit{Complex Fibonacci numbers and Fibonacci quaternions.} The American Mathematical Monthly 70.3, (1963):289-291.

\bibitem{ 14 Hor} Horadam, A. F. \textit{Generating functions for powers of a certain generalized sequence of numbers.} Duke Math. J 32, (1965): 437-446.

\bibitem{ 15 Horadam2} Horadam, A. F. \textit{Special Properties of the Sequence $W_n$ $(a,b;p,q)$.} Fibonacci. Quart. 5(5), (1967):424-434.

\bibitem{ 16 ke} Ke\c{c}ilio\u{g}lu O. and I. Akku\c{s} \textit{The fibonacci octonions.} Advances in Applied Clifford Algebras 25.1 (2015): 151-158.

\bibitem{ 17 kos} Koshy, T. Fibonacci and Lucas numbers with applications. Vol. 51. John Wiley and Sons, 2011.

\bibitem{ 18 le} Lewis, D. W. \textit{Quaternion algebras and the algebraic legacy of Hamilton’s quaternions.} Irish Math. Soc. Bull 57 (2006): 41-64.

\bibitem{ 19 Loune} Lounesto, P. \textit{Clifford Algebras and Spinors} vol. 286. Cambridge university press, Cambridge 2001.

\bibitem{ 20 Savi} Savin, D. \textit{Some properties of Fibonacci numbers, Fibonacci octonions, and generalized Fibonacci-Lucas octonions.} Advances in Difference Equations 2015.1 (2015): 1.

\bibitem{ 21 Szyn} Szynal-Liana, A. and I. Wloch \textit{The Pell quaternions and the Pell octonions.} Advances in Applied Clifford Algebras 26.1, (2016):435-440.

\bibitem{ 22 Unal} \"{U}nal, Z., \"{U}. Toke\c{s}er and G. Bilgici \textit{Some Properties of Dual Fibonacci and Dual Lucas Octonions.} Advances in Applied Clifford Algebr. doi:10.1007/s00006-016-0724-4

\end{thebibliography}
\end{document}